\documentclass[11pt]{amsart}
\usepackage[utf8]{inputenc}
\usepackage{amsmath} 
\usepackage{amssymb} 
\usepackage{amsthm,amsrefs} 
\usepackage{amsfonts} 
\usepackage{mathtools}
\usepackage{pgfplots}
\usepackage{tikz}
\usepackage{cases}
\usepackage{enumerate}
\usetikzlibrary{positioning}
\usetikzlibrary{angles,quotes}

\pgfplotsset{compat=1.18, width=12cm}
\usepackage[hidelinks,colorlinks=true,linkcolor=blue,citecolor=blue]{hyperref}
\usepackage[margin= 1in]{geometry}
\usepackage{tabularx}
\usepackage{ulem}
\usepackage{graphicx}

\renewcommand{\phi}{\varphi}

\usepackage[capitalise]{cleveref}

\usepackage{pythonhighlight}

\newcommand{\mb}[1]{\mathbb{#1}}

\newcommand*{\rom}[1]{\expandafter\romannumeral #1}

\newcommand{\mm}{\mathbb{R}}
\newcommand{\m}{t}
\newcommand{\mmone}{t-1}

\newtheorem{theorem}{Theorem}

\newtheorem{conjecture}[theorem]{Conjecture}
\newtheorem{lemma}[theorem]{Lemma}

\newtheorem{proposition}[theorem]{Proposition}
\newtheorem{remark}[theorem]{Remark}

\title{Illumination number of 3-dimensional cap bodies}

\author{Andrii Arman}
\address{Department of Mathematics, University of Manitoba, Winnipeg, MB, R3T 2N2, Canada}
\email{andrew0arman@gmail.com}
\thanks{}

\author{Jaskaran Singh Kaire}
\email{singhj82@myumanitoba.ca}
\thanks{The second author was supported in part by NSERC of Canada Discovery Grant RGPIN-2020-05357.}

\author{Andriy Prymak}
\email{prymak@gmail.com}
\thanks{The third author was supported by NSERC of Canada Discovery Grant RGPIN-2020-05357.}

\date{\today}

\keywords{Illumination problem, Hadwiger's covering conjecture, cap body, random rotations, integer linear programming.}

\subjclass[2020]{Primary 52A15; Secondary 52A40, 52C17, 90C05}

\begin{document}
\begin{abstract}

The illumination conjecture asserts that any convex body in $n$-dimensional Euclidean space can be illuminated by at most $2^n$ external light sources or parallel beams of light. Despite recent progress on the illumination conjecture, it remains open in general, as well as for specific classes of bodies. 

Bezdek, Ivanov, and Strachan showed that the conjecture holds for symmetric cap bodies in sufficiently high dimensions. Further, Ivanov and Strachan calculated the illumination number for the class of 3-dimensional centrally symmetric cap bodies to be 6.

In this paper, we show that even the broader class of all 3-dimensional cap bodies has the same illumination number 6, in particular, the illumination conjecture holds for this class. The illuminating directions can be taken to be vertices of a regular tetrahedron, together with two special directions depending on the body. The proof is based on probabilistic arguments and integer linear programming.

\end{abstract}

\maketitle

\section{Introduction}

Hadwiger~\cite{H} asked the following question in 1957: For $n\geq 3$, what is the smallest number $N(n)$ such that every $n$-dimensional convex body can be covered by the union at most $N(n)$ of translates of the body's interior? The earliest formulation of the question dates back to 1955~\cite{LE}, where Levi proved that $N(2)=4$.
After Hadwiger and Levi, the question was restated by
Gohberg and Markus~\cite{GM} (independently without knowing about the work of Levi and
Hadwiger) in terms of covering by homothetic copies. 
For more details about the history of the problem and an extensive list of partial results, see, e.g.,~\cite{BK}. 

The example of the $n$-dimensional cube yields $N(n)\ge 2^n$. It is widely believed that $N(n)=2^n$ for all $n\ge 2$. 
\begin{conjecture}\label{conj}
    For every $n\geq 3$, any $n$-dimensional convex body can be covered by $2^n$ or fewer smaller positive homothetic copies of itself. Furthermore, $2^{n}$ homothetic copies are required only if the body is an affine copy of the $n$-cube. 
\end{conjecture}

Boltyanski~\cite{B} showed that the conjecture is equivalent to a certain illumination problem.
For a convex body $K$,
a direction (unit vector) $v$ illuminates a point $x$ on the boundary $\partial K$ of $K$, if the ray $\{x+vt \;:\; t\geq 0\}$ has nonempty intersection with interior of $K$. The set of directions 
$\{v_i\}_{i=1}^k$ is said to illuminate $K$ if every point of $\partial K$ is illuminated by some $v_i$. The illumination number $I(K)$ of $K$ is the smallest $k$ for which $K$ can be illuminated by $k$ directions. Boltyanski~\cite{B} showed that ~\cref{conj} is equivalent to the following: the illumination number of a convex body in $\mm{}^n$ does not exceed $2^n$, with equality only for affine copies of the $n$-cube.

In literature, \cref{conj} has been referred to as the Hadwiger Covering Conjecture, the Levi-Hadwiger(-Gohberg-Markus) Conjecture, and Hadwiger-Boltyanski Illumination Conjecture. 



As of today, the conjecture has only been solved for the $2$-dimensional case, and it is believed to be notoriously hard even for the next smallest dimension, namely, $\mm^3$. Papadoperakis~\cite{P} proved that the illumination number of any convex body in ${\mm{}}^3$ is at most 16, and Prymak~\cite{Pr} improved the bound to 14. This result is far from the expected illumination number $8$, but it is the best known upper bound on illumination number in $\mm{}^3$. 

The conjecture has been verified for certain classes of bodies in $\mm{}^3$. Bezdek and Kiss~\cite{BKiss} showed that the illumination number of almost smooth convex bodies and the illumination number of convex bodies of constant width is at most 6. The result about bodies of constant width has been proven independently by several groups of researchers (refer to~\cite{BLNP} for more details). Martini~\cite{M} proved that the illumination number of every zonotope in $\mm^3$, which is not a parallelotope, is at most 6.  Lassak \cite{L} proved that if the convex body is centrally symmetric, then the conjecture holds in $\mm{}^3$. Furthermore, Dekster \cite{De} proved (without treatment of the equality case) that the conjecture holds in $\mm{}^3$ if the convex body is symmetric about a plane.


Let $\mb{B}^n$ be the closed unit ball of $\mm^n$ centred at the origin, and $\text{conv}(\cdot)$ denotes the closed convex hull. $K\subset\mm^n$ is a cap body (of $\mb{B}^n$) if and only if $K=\text{conv}(\mb{B}^n\cup V)$ for a set $\{x_i\}_{i=1}^m\subset\mm^n\setminus\mb{B}^n$ of ``vertices'' such that the closed segment joining any two distinct $x_i,x_j$, $1\le i<j\le m$, intersects $\mb{B}^n$. Alternatively, a cap body is the union of the ball $\mb{B}^n$ and the cone-shaped ``spikes'' $\text{conv}(\{x_i\} \cup \mb{B}^{n})\setminus \mb{B}^n$, $1\le i\le m$, which have mutually disjoint interiors, see \cref{fig:capbody} in the next section for an example. 

Cap bodies were considered by Nasz\'odi in~\cite{Na} (called spiky balls there) as examples of convex bodies which can be ``very close'' to the ball while having exponentially large (w.r.t. dimension $n$) illumination numbers. Bezdek, Ivanov and Strachan~\cite{BIS} proved Conjecture~\ref{conj} for centrally symmetric cap bodies in sufficiently large dimensions. In particular, they proved that any $n$-dimensional centrally symmetric cap body can be illuminated by $<2^n$ directions in Euclidean $n$-space for $n=3,$ $4,$ $9$ and $n\geq 19$. Moreover, Ivanov and Strachan~\cite{IS} proved that any $3$-dimensional centrally symmetric cap body can be illuminated by $6$ directions. This bound cannot be improved: the convex hull of the unit ball and an appropriately scaled regular octahedron is a centrally symmetric cap body with illumination number exactly $6$. In this work we show that for any convex body from the much larger class of not necessarily symmetric $3$-dimensional cap bodies, the illumination number still does not exceed $6$. This is sharp by the same example.

\begin{theorem}\label{thm:main}
    For any cap body $K\subset\mm{}^3$, $I(K)\leq 6$.
\end{theorem}

In particular, \cref{thm:main} verifies the conjecture for all cap bodies in $\mm^3$. The illuminating directions can be taken to be vertices of a regular tetrahedron, together with two special directions depending on the body. In particular, we show, using integer linear programming,  that for a random rotation of vertices of regular tetrahedron, the maximum (taken over all cap bodies) of expected number of un-illuminated caps is strictly less than 3.

\section{Preliminaries}


Let $\mm{}^n$ be equipped with the standard Euclidean inner product $\langle \cdot,\cdot \rangle$, and the norm $\|\cdot\|$.  The unit sphere centred at the origin is denoted by $\mb{S}^{n-1}:=\{x\in \mm{}^n: \|x\|=1\}$, and the unit ball by $\mb{B}^{n}=\{x\in \mm{}^n: \|x\|\leq 1\}$. For any points $x,y\in \mb{S}^{n-1}$, the geodesic distance between them is $\theta(x,y):=\arccos\langle x,y\rangle$. For $\xi \in \mb{S}^{n-1}$, define the open and closed caps on $\mb{S}^{n-1}$ centred at $\xi$ of radius $\varphi$ by
$$C(\xi,\varphi):=\{y \in \mb{S}^{n-1}: \langle \xi, y \rangle> \cos\varphi\}, \quad \quad C[\xi,\varphi]:=\{y \in \mb{S}^{n-1} : \langle \xi, y \rangle\geq \cos\varphi\}.$$

Recall that $\text{conv}(\cdot)$ is the closed convex hull. $K$ is a convex body in $\mm{}^n$ if $K\subset\mm^n$ is a convex compact set with non-empty interior. We now repeat the definition of the subclass of cap bodies. $K\subset\mm^n$ is a cap body if and only if $K=\text{conv}(\mb{B}^n\cup V)$ for a set $\{x_i\}_{i=1}^m\subset\mm^n\setminus\mb{B}^n$ of ``vertices'' such that the closed segment joining any two distinct $x_i,x_j$, $1\le i<j\le m$, intersects $\mb{B}^n$. Cap bodies are obtained by taking the union of $\mb{B}^n$ with  ``spikes'' $\text{conv}(\{x_i\} \cup \mb{B}^{n})\setminus \mb{B}^n$, $1\le i\le m$. Convexity of the resulting union is equivalent to the property that the spikes have mutually disjoint interiors, or to the above stated property that the segments joining distinct vertexes intersect $\mb{B}^n$. An example of a cap body is given in \cref{fig:capbody}.
\begin{figure}[h]
    \centering
    \begin{tikzpicture}[scale=2]
    
        \draw (0,0) circle(1);

        \coordinate [label=left:$x_1$] (A) at (-1.25,1.35);
        \draw (-1.25,1.25) -- (0.225,0.975);
        \draw (-1.25,1.25) -- (-0.99,-0.14);

        \coordinate [label=left:$x_2$] (B) at (1.85, 0.29273);
        \draw (1.50719, 0.29273) -- (0.52,0.85);
        \draw (1.50719, 0.29273) -- (0.82,-0.57);

        \draw[dashed] (-1.25,1.25)--(1.50719, 0.29273);

        \node at (0,0) (o1) {};
        \fill[black] (o1) circle (1pt) node {};
        \node at (0.15,-0.15) {$O$};
        
        \node at (-0.4,-0.15) {$\mb{B}^2$};
    \end{tikzpicture}
    \caption{A cap body in $\mb{R}^2$ with 2 vertices $x_1,x_2$.}
    \label{fig:capbody}
\end{figure}


For a given vertex $x_i$ of a cap body, the corresponding base cap (or simply cap) is defined to be the set 
\[
S_i:=\overline{\text{conv}(\{x_i\}\cup \mb{B}^{n})\setminus \mb{B}^n}\cap \mb{S}^{n-1},
\]
where $\overline{(\cdot)}$ denotes the closure of the set (see \cref{fig:cap}).

\begin{figure}[h]

\begin{tikzpicture}[scale=2,line cap=round,line join=round,x=1cm,y=1cm]

\hspace{-1.35cm}

\clip(-4.667072743468536,-1.013299733565548) rectangle (4.814020724982975,2.232276514659644);
\draw [shift={(0,0)},line width=1pt,color=black,fill=black,fill opacity=0.10000000149011612] (0,0) -- (90:0.17134506268285862) arc (90:150:0.17134506268285862) -- cycle;
\draw [line width=1pt] (0,0) circle (1cm);
\draw [line width=1pt] plot[domain=4.1887902047863905:5.235987755982988,variable=\t]({1.7320508075688772*cos(\t r)},{1.7320508075688772*sin(\t r)+2});
\draw [line width=1pt, dotted]  plot[domain=0.523598776:2.61799,variable=\t] ({cos(\t r)},{0.5+0.35*sin(\t r)-0.35*0.5});
\fill [color=gray, opacity=0.2, variable=\x]
(-0.86602540,0.5)
--plot[domain=4.1887902047863905:5.235987755982988,variable=\t]({1.7320508075688772*cos(\t r)},{2+1.7320508075688772*sin(\t r)})
--(0.86602540,0.5)
--plot[domain=0.523598776:2.61799,variable=\t] ({cos(\t r)},{sin(\t r)})
--(-0.86602540,0.5);
\fill [color=gray, opacity=0.2, variable=\x]
(-0.86602540,0.5)
--plot[domain=2.61799:0.523598776,variable=\t] ({cos(\t r)},{0.5+0.35*sin(\t r)-0.35*0.5})
--(0.86602540,0.5)
--plot[domain=0.523598776:2.61799,variable=\t] ({cos(\t r)},{sin(\t r)})
--(-0.86602540,0.5);
\draw [line width=1pt] (0,2)-- (-0.8660254037844386,0.5);
\draw [line width=1pt] (0,2)-- (0.8660254037844387,0.5);
\draw [line width=1pt, dotted] (-0.8660254037844386,0.5)-- (0,0);
\draw [line width=1pt, dotted] (0,0)-- (0,2);
\draw [fill=black] (0,0) circle (1pt);
\draw [fill=black] (0,2) circle (1pt);
\draw[color=black] (0.04491648031007679,2.124196395875493+0.01) node {$x_i$};
\draw [fill=black] (0,1) circle (1pt);
\draw[color=black] (0.04491648031007679+0.1,1.1132605260466277+0.1) node {$\widehat{x_i}$};
\draw[color=black] (0.20483853881407818-0.4,0.12517066457547704-0.2) node {$\phi_i $};
\draw[color=black] (0.20483853881407818-0.1,0.12517066457547704-0.2) node {$O$};
\draw[color=black] (0.866+0.5,0.5+0.2) node {$S_i=C[\widehat{x_i},\phi_i]$};
\end{tikzpicture}
\caption{A base cap $S_i$ with its center $\widehat{x_i}$.}
\label{fig:cap}
\end{figure}

Note that $S_i = C[\widehat x_i,\varphi_i]$, where $\widehat{x}_i:=\frac{x_i}{\|x_i\|}$ is the centre of the cap, and $\varphi_i=\arccos{}\frac{1}{\|x_i\|}$ is the radius of the cap. It is an easy observation that the base caps always have acute radii, i.e., $\varphi_i<\pi/2$ for $1\le i\le m$. Furthermore, observe that the convexity of $K$ implies that $C(\widehat x_\alpha,\varphi_\alpha)\cap C(\widehat x_\beta,\varphi_\beta)=\emptyset$ for any distinct $x_\alpha, x_\beta\in \{x_i\}_{i=1}^m$. One can construct a cap body from any given system of mutually non-overlapping open caps of radius $<\tfrac\pi2$ on the sphere. In particular, the vertex $x_i$ can be expressed using $C(\widehat x_i,\varphi_i)$ as follows: \[x_i=\frac{\widehat x_i}{\cos{\varphi_i}}.\] We use the following proposition from~\cite{BIS,ABPR}:


\begin{proposition}
    \label{prop:ABPR} A cap body $K$ with vertices $\{x_i\}_{i=1}^m$ is illuminated by the directions $\{v_j\}^k_{j=1} \subset \mb{S}^{n-1}$ if:
    \begin{enumerate}[i)]
    \item $C\left(-\widehat{x}_i,\frac{\pi}{2}-\varphi_i\right)\cap \{v_j\}^k_{j=1} \neq \emptyset$ for each $i$, $1\leq i\leq m$,
    \item the positive hull of $\{v_i\}^k_{i=1}$ is $\mm{}^n$, i.e. for all $x
    \in \mm{}^n$ there are positive $c_1,\ldots,c_k$ such that $x=c_1v_1+\cdots+c_kv_k$. 
\end{enumerate}
\end{proposition}
We will say that the cap $S_i$ is illuminated by $v_j$ if and only if $v_j\in C\left(-\widehat{x}_i,\frac{\pi}{2}-\varphi_i\right)$.

\section{Proof of \cref{thm:main}}


First, let us prove the existence of a cap body $K_0$ such that $I(K_0)=6$. This was already established by Ivanov and Stranchan~\cite{IS}, but we include the example here for the sake of completeness. Define $S:=\{\pm (1,0,0), \pm (0,1,0),\pm (0,0,1)\}$, and construct $K_0$ by taking $6$ caps (open caps) with radius $\frac{\pi}{4}$ and centers at each element of $S$.
As per \cref{prop:ABPR}, we need to find a set $V=\{v_i\}_{i=1}^k$ such that $C\left(-x,\frac{\pi}{4}\right)\cap V \neq \emptyset$ for each $x\in S$, and positive convex hull of $\{v_i\}^k_{i=1}$ is $\mm^3$. 
Note that the caps $C\left(-x,\frac{\pi}{4}\right), \; x\in S$ do not intersect, so each cap requires at least one direction. So, $|V|\geq 6$. It is easy to see that $V=S$ illuminates the body, as positive hull of $S$ is $\mm^3$. This proves that $I(K_0) = 6$.



Now, suppose $K$ is a cap body with vertices $x_1, x_2, \dots, x_m$. Recall that for a given $x_i$, the corresponding base cap has a central angle $\phi_i=\arccos(1/\|x_i\|)$ (see \cref{fig:cap}), 
and we assume that sizes of caps are ordered $\varphi_1\geq \varphi_2 \geq \ldots \geq \varphi_m$. We will show $I(K)\le 6$.




Let $L$ be a set of vertices of a regular tetrahedron inscribed into the unit sphere. It is routine work to show that $L$ satisfies \cref{prop:ABPR} (ii), and thus we only need to check \cref{prop:ABPR} (i) to complete the proof. 

For every $\theta\in(0,\frac{\pi}{2}]$, define $C_\theta:=\bigcup_{l\in L}C[l,\theta]$.
It is well-known and is a simple computation that $C_\theta=\mb{S}^2$ when $\theta \ge \arccos\left({\frac{1}{3}}\right)$. Therefore, by \cref{prop:ABPR}~(i), any vertex $x_j$ of $K$ with $\varphi_j<\frac\pi2-\arccos\left({\frac{1}{3}}\right)$ is illuminated by one of the directions from any rotation of $L$. Thus, in what follows, we assume that $\varphi_m\ge \frac\pi2-\arccos\frac13$.

Let $\sigma$ be the probabilistic spherical measure on $\mb{S}^2$. To find a suitable random rotation that illuminates most of the caps, we need to know the value of $\sigma(C_\theta)$.
    \begin{lemma}
        \label{lem:prob_comp}
        We have:
        \begin{numcases}{\sigma(C_\theta)=}
            2(1-\cos\theta), & $0<\theta\le \tfrac12\arccos(-\tfrac13)$, \label{eqn:case no overlap}\\
            2(1-\cos\theta)-6A_\theta, & $\tfrac12\arccos(-\tfrac13)\le \theta<\arccos\tfrac13$, \label{eqn:case overalap}\\
            1, & $\arccos\tfrac13\le\theta\le\tfrac\pi2$, \label{eqn:case whole sphere}
        \end{numcases}
        where
        \begin{equation*}
        A_\theta=\frac{1}{2\pi}\left(-\arccos\left(\frac{-\frac{1}{3}-\cos^2\theta}{\sin^2\theta}\right)-2\arccos\left(\sqrt{2}\cot\theta\right)\cos\theta+\pi\right).
        \end{equation*}
    \end{lemma}
    \begin{proof}
        Recall that $C_\theta=\mb{S}^2$ for $\theta\geq \arccos(1/3)$, so~\eqref{eqn:case whole sphere} readily follows. \eqref{eqn:case no overlap} is straightforward from the standard $\sigma(C[x,\theta])=\frac{1-\cos\theta}{2}$ and the fact that the four congruent caps in $C_\theta$ do not overlap in this case. For~\eqref{eqn:case overalap}, we observe that there will be six congruent ``lunes'' which are intersections of two caps from $C_\theta$. The measure $A_\theta$ of each lune can be computed using, for example~\cite{TV}*{Eqs.~(2)--(4)}.
    \end{proof}
    
Next, we will upper bound the expected number of caps not illuminated by a random rotation of $L$. This will be done utilizing integer linear programming. We need to consider various cases depending on the number of caps of different sizes. Set $a_0:=\tfrac{19\pi}{180}<\tfrac\pi2-\arccos\tfrac13\le\phi_m$. For a suitable positive integer $t$ that will be selected later, we define a discretization array $a=[a_0, a_1, \dots, a_t]$, where $a_t:=\frac\pi2$ and $a_i=a_0+i\frac{a_t-a_0}{t}$ for $1\leq i\leq t-1$ are equally spaced on $[a_0,a_t]$. We have $\phi_j\in(a_0,a_t]$ for any $j$. Let $n_i$, $0\le i\le \mmone$, denote the number of indices $j$ such that $\varphi_j \in(a_i, a_{i+1}]$. Note that by \cref{prop:ABPR}~(i) a vertex $x_j$ with $\varphi_j \in(a_i, a_{i+1}]$ is illuminated by a random rotation $L'$ of $L$, provided at least one of the points of $L'$ is within the geodesic distance $\tfrac\pi2-\phi_j$ of  $-\widehat{x_j}$. Therefore, the probability that $x_j$ is not illuminated is $1-\sigma\left(C_{\tfrac\pi2-\phi_j}\right)\le 1-\sigma\left(C_{\tfrac\pi2-a_{i+1}}\right)$. Then overall, the expected number of caps which are not illuminated does not exceed
\begin{equation}\label{eqn:target fn}
\sum_{i=0}^{t-1} n_i \left(1-\sigma\left(C_{\tfrac\pi2-a_{i+1}}\right) \right),
\end{equation}
which will be our target function in the integer programming problem. 

Due to the convexity of $K$, we know that the caps $C(\widehat{x}_j,\phi_j)$ do not overlap, so the total measure of these caps is at most $1$. In terms of $n_i$, this provides the following constraint:
\begin{equation}
    \label{eqn:const bases pack}
    \sum_{i=0}^{t-1} n_i\frac{1-\cos a_i}2\le 1.
\end{equation}

To obtain another constraint, we need the following lemma which states that one cannot pack five caps of radius $>\tfrac\pi4$ on $\mb{S}^2$. This follows from, e.g.~\cite{Ra}*{Th.~1~(iii)}, but we include the short proof anyway for completeness.

\begin{lemma} 
\label{lemma:caps_pi/4} If $m\ge 5$, then $\varphi_5\le \frac{\pi}{4}$.
\end{lemma}

\begin{proof}
    Assume to the contrary that $\varphi_5>\frac{\pi}{4}$, hence $\varphi_j>\frac\pi4$, $1\le j\le 5$. 
    As $K$ is a cap body, all base caps $C(\widehat{x}_j,\varphi_j)$, $1\le j\le 5$, are disjoint, so $\theta(x_i,x_j)\geq \phi_i+\phi_j>\frac{\pi}{2},$ and so $\langle \widehat{x}_i,\widehat{x}_j\rangle<0$, $1\le i<j\le 5$. The points $\widehat{x}_1, \widehat{x}_2,\dots, \widehat{x}_5$ are affinely dependent, so there are $c_1$, $c_2$, $c_3$, $c_4$, $c_5$ (not all zero) such that $$c_1\widehat{x}_1+\cdots+c_5\widehat{x}_5=0$$ and $\sum_{i=1}^5c_i=0$. Suppose that $I_+=\{c_i : c_i\geq 0\}$ and $I_-=\{c_j : c_j< 0\}$, now
    $$\sum_{c_i\in I_+} c_i\widehat{x}_i = \sum_{c_j\in I_-} -c_j\widehat{x}_j.$$
    Finally, taking the dot product with $\sum_{c_i\in I_+} c_i\widehat{x}_i$ on both sides we get
    $$0\leq \left\|\sum_{c_i\in I_+} c_i\widehat{x}_i\right\|^2=\sum_{c_i\in I_+} c_i\widehat{x}_i\sum_{c_j\in I_-} -c_j\widehat{x}_j=\sum c_i(-c_j)\langle\widehat{x}_i,\widehat{x}_j\rangle<0,
    $$
    which is the desired contradiction.
\end{proof}

In terms of $n_i$, \cref{lemma:caps_pi/4} implies that
\begin{equation}
    \label{eqn:const large caps}
    \sum_{0\le i<\m\,:\, a_i\ge \tfrac\pi4}n_i\le 4.
\end{equation}

    Let $M_t$ denote the solution of the integer linear programming problem 
    \begin{equation*}
        \text{\it maximize~\eqref{eqn:target fn} subject to~\eqref{eqn:const bases pack} and~\eqref{eqn:const large caps}}
    \end{equation*}
    with non-negative integer variables $n_i$, $0\le i\le t-1$. If for some $t$ we get $M_t<3$, then for every cap body $K$ there exists a rotation of $L$ such that the number of caps not illuminated by this rotation is at most $\lfloor M_t \rfloor\le 2$. Assigning a direction for each of the un-illuminated caps, we obtain an illuminating system with at most $6$ directions.

    For any given $t$, the computation of $M_t$ is a standard integer linear program for which many solvers are available. We are interested in an upper bound on $M_t$. Thus, to avoid any numerical errors, we can fix a positive integer denominator $D$ and round up the coefficients (computed symbolically) in \eqref{eqn:target fn}, 
    and round down the coefficients in \eqref{eqn:const bases pack} to the nearest number from  $\frac1D\mb{Z}$. The obtained modified integer linear program will have rational coefficients and can be solved precisely without numerical errors. The corresponding SageMath~\cite{sagemath} script is given in the appendix of~\cite{v1}. With $D=3000$ and $t=250$, the solution to the modified problem is $\frac{2999}{1000}$, and thus $M_{250}\le 2.999$ implying the desired $\lfloor M_{250} \rfloor\le 2$ and completing the proof.
    

    \begin{remark}
        The script takes less than 2 hours to complete on a modern personal computer. If we switch to computations with floating point arithmetics, the approximate solution can be obtained significantly faster (under 1 second for the same $t=250$). Such computations with much larger $t$ suggest that $\limsup_{t\to\infty} M_t<2.97$. We emphasize that the bound $M_{250}\le 2.999$ in our proof was obtained using symbolic computations only. One can also get the required bound using fewer intervals than $250$ by partitioning $[a_0,\frac{\pi}{2}]$ in a non-uniform manner. We opted to use the uniform partition and simplicity of the exposition at the cost of a slight increase in computation time.
    \end{remark}

    \begin{remark}
        If we aim to show that $I(K)\leq 7$, then a human verifiable proof is possible. The proof would require adding one geometric argument and would make the optimization problem reducible to under a dozen cases. 
    \end{remark}
    
    

\end{document}